\UseRawInputEncoding
\documentclass[12pt]{article}

\usepackage{dsfont}
\usepackage{amsfonts,amsmath,amsthm}
\usepackage{amssymb}
\usepackage{algorithm}
\usepackage{algpseudocode}
\usepackage{xcolor}
\usepackage{graphicx}
\usepackage{stmaryrd}        
\usepackage{multirow}

\usepackage[paper=a4paper,dvips,top=2cm,left=2cm,right=2cm,
foot=1cm,bottom=4cm]{geometry}

\newcommand{\blue}[1]{\begin{color}{blue}#1\end{color}}

\newcommand{\ve}{{\bf e}}

\newcommand{\vd}{{\bf d}}

\begin{document}
	\Large
	
	\title{Completely Positive Biquadratic Tensors}
	\author{Liqun Qi\footnote{Jiangsu Provincial Scientific Research Center of Applied Mathematics, Nanjing 211189, China.
			Department of Applied Mathematics, The Hong Kong Polytechnic University, Hung Hom, Kowloon, Hong Kong.
			({\tt maqilq@polyu.edu.hk})}
		\and
		Chunfeng Cui\footnote{School of Mathematical Sciences, Beihang University, Beijing  100191, China.
			({\tt chunfengcui@buaa.edu.cn})}
		\and
		Haibin Chen\footnote{School of Management Science, Qufu Normal University, Rizhao, Shandong  276800, China.
			({\tt chenhaibin508@qfnu.edu.cn})}
		\and {and \
			Yi Xu\footnote{School of Mathematics, Southeast University, Nanjing  211189, China. Nanjing Center for Applied Mathematics, Nanjing 211135,  China. Jiangsu Provincial Scientific Research Center of Applied Mathematics, Nanjing 211189, China. ({\tt yi.xu1983@hotmail.com})}
		}
	}

	\date{\today}
	\maketitle
	
	\begin{abstract}
		In this paper, we systemically introduce completely positive biquadratic (CPB) tensors and copositive biquadratic tensors.   We show that all weakly CPB tensors are sum of squares tensors, the CPB tensor cone and the copositive biquadratic tensor cone are dual to each other.    We also show that the outer product of two completely positive matrices is a CPB tensor, and the outer product of two copositive matrices is a copositive biquadratic tensor.   We then study two easily checkable subclasses of CPB tensors, namely positive biquadratic Cauchy tensors and biquadratic Pascal tensors.  We show that a biquadratic Pascal tensor is both strongly CPB and positive definite.

		\medskip

		\medskip

		
	\end{abstract}

	\renewcommand{\Re}{\mathds{R}}
	\newcommand{\rank}{\mathrm{rank}}
	\newcommand{\X}{\mathcal{X}}
	\newcommand{\A}{\mathcal{A}}
	\newcommand{\I}{\mathcal{I}}
	\newcommand{\B}{\mathcal{B}}
	\newcommand{\PP}{\mathcal{P}}
	\newcommand{\C}{\mathcal{C}}
	\newcommand{\D}{\mathcal{D}}
	\newcommand{\LL}{\mathcal{L}}
	\newcommand{\OO}{\mathcal{O}}
	\newcommand{\e}{\mathbf{e}}
	\newcommand{\0}{\mathbf{0}}
	\newcommand{\1}{\mathbf{1}}
	\newcommand{\dd}{\mathbf{d}}
	\newcommand{\ii}{\mathbf{i}}
	\newcommand{\jj}{\mathbf{j}}
	\newcommand{\kk}{\mathbf{k}}
	\newcommand{\va}{\mathbf{a}}
	\newcommand{\vb}{\mathbf{b}}
	\newcommand{\vc}{\mathbf{c}}
	\newcommand{\vq}{\mathbf{q}}
	\newcommand{\vg}{\mathbf{g}}
	\newcommand{\pr}{\vec{r}}
	\newcommand{\pc}{\vec{c}}
	\newcommand{\ps}{\vec{s}}
	\newcommand{\pt}{\vec{t}}
	\newcommand{\pu}{\vec{u}}
	\newcommand{\pv}{\vec{v}}
	\newcommand{\pn}{\vec{n}}
	\newcommand{\pp}{\vec{p}}
	\newcommand{\pq}{\vec{q}}
	\newcommand{\pl}{\vec{l}}
	\newcommand{\vt}{\rm{vec}}
	\newcommand{\x}{\mathbf{x}}
	\newcommand{\vx}{\mathbf{x}}
	\newcommand{\vy}{\mathbf{y}}
	\newcommand{\vu}{\mathbf{u}}
	\newcommand{\vv}{\mathbf{v}}
	\newcommand{\y}{\mathbf{y}}
	\newcommand{\vz}{\mathbf{z}}
	\newcommand{\T}{\top}
	\newcommand{\R}{\mathcal{R}}
	\newcommand{\Q}{\mathcal{Q}}
	\newcommand{\TT}{\mathcal{T}}
	\newcommand{\Sc}{\mathcal{S}}
	\newcommand{\N}{\mathbb{N}}	
	
	\newtheorem{Thm}{Theorem}[section]
	\newtheorem{Def}[Thm]{Definition}
	\newtheorem{Ass}[Thm]{Assumption}
	\newtheorem{Lem}[Thm]{Lemma}
	\newtheorem{Prop}[Thm]{Proposition}
	\newtheorem{Cor}[Thm]{Corollary}
	\newtheorem{example}[Thm]{Example}
	\newtheorem{remark}[Thm]{Remark}
	
	\section{Introduction}
	
	Completely positive tensors and copositive tensors \cite{LQ16, QL17, QXX14} are extensions of completely positive matrices and copositive matrices \cite{BS03}, and have wide applications in numerical optimization and hypergraph theory \cite{CW18, Ko15, NYZ18, PVZ15, YXH22, ZF18, ZF19, ZFW20}.  On the other hand, biquadratic tensors have wide applications in statistics, mechanics, relative theory, bipartite hypergraph theoy and polynomial theory \cite{Ca73, CB25, CHHS25, Ch75, CQX25, HLW20, LLL19, QC25, QCX25, QDH09, WSL20, Zh23, ZLS24}.
	
	In this paper, we will systematically introduce completely positive biquadratic (CPB) tensors and copositive biquadratic tensors, and study their properties.
	
	In the next section, we define CPB tensors and copositive biquadratic tensors.  We show that the CPB tensor cone and the copositive biquadratic tensor cone are dual to each other.  In particular, we show that all weakly CPB tensors, thus all CPB tensors are sum of squares tensors.
	
	In Section 3, we show that the outer product of two completely positive matrices is a CPB tensor, and the outer product of two copositive matrices is a copositive biquadratic tensor.
	
	The, in Section 4, we study two easily checkable subclasses of CPB tensors, namely positive biquadratic Cauchy tensors and biquadratic Pascal tensors.  We show that a biquadratic Pascal tensor is both strongly CPB and positive definite.
	
	\section{Completely Positive Biquadratic Tensors and Copositive Biquadratic Tensors}
	
	Let $[m] = {\{1, \dots, m\}}$. Denote the $i$th unit vector of $\Re^m$ by $\ve_i^{(m)}$.
	
	An $m \times n$ biquadratic tensor $\A = (a_{ijkl})$ has real entries $a_{ijkl}$ for $i, k \in [m]$ and $j, l \in [n]$.
	Assume that $\A$ is symmetric, i.e.,
	$$a_{ijkl} = a_{kjil} = a_{ilkj}$$
	for $i, k \in [m]$ and $j, l \in [n]$.
	
	Let $\x \in \Re^m$ and $\y \in \Re^n$.  Then
	$$F_\A(\x, \y) \equiv \A\x\y\x\y := \sum_{i, k=1}^m \sum_{j, l=1}^n a_{ijkl}x_iy_jx_ky_l.$$
	If $\A\x\y\x\y \ge 0$ for all $\x \in \Re^m$, $\y \in \Re^n$, then we say that $\A$ is {\bf positive semidefinite (psd)}.  If furthermore, $\A\x\y\x\y > 0$ for all $\x \in \Re^m$, {$\y \in \Re^n$,} satisfying $\|\x \|_2 = \|\y\|_2 = 1$, then we say that $\A$ is {\bf positive definite (pd)}.
	
	We call the homogeneous polynomial $F_\A(\x, \y)$ an $m \times n$ {\bf biquadratic form}.  If $F_\A(\x, \y)$ can be expressed as the sum of squares of {\bf bilinear forms}:
	$$F_\A(\x, \y) = \sum_{{r}=1}^R \left[f_{{r}}(\x, \y)\right]^2,$$
	where $f_{{r}}(\x, \y), {r}\in [R]$ are bilinear forms, then $\A$ is called an {\bf SOS biquadratic tensor}.   The smallest $R$ is called the {\bf SOS rank} of $\A$.     Clearly, if $\A$ is an SOS biquadratic tensor, then $\A$ is psd, but not vice versa in general.  In 1973, Calder\'{o}n \cite{Ca73} proved that an $m \times 2$ psd biquadratic tensor is always an SOS biquadratic tensor.   In 1975, Choi \cite{Ch75} gave an example that a $3 \times 3$ psd biquadratic tensor which is not an SOS biquadratic tensor.    Recently, we proved in \cite{CQX25} that the SOS rank of a $2 \times 2$ psd biquadratic tensor is at most $3$.   
	
	If $\A\x\y\x\y \ge 0$ for all $\x \in \Re^m_+$, $\y \in \Re^n_+$, then $\A$ is called a {\bf copositive {biquadratic} tensor}.
	If $\A\x\y\x\y > 0$ for all $\x \in \Re^m_+$, $\x \not = \0_m$, $\y \in \Re^n_+$, $\y \not = \0_n$, then $\A$ is called a {\bf strictly copositive {biquadratic} tensor}.
	
	All $m \times n$ symmetric biquadratic tensors form a space, {denoted by} $SBQ(m, n)$.  Then all the $m \times n$ symmetric positive semidefinite biquadratic tensors form a closed cone, denoted as $PSDSBQ(m, n)$.   All the $m \times n$ symmetric copositive biquadratic tensors also form a closed cone, denoted as $COPSBQ(m, n)$.   We have
	$$PSDSBQ(m, n) \subsetneq  COPSBQ(m, n) \subsetneq SBQ(m, n).$$
	
	Suppose that $\A = (a_{ijkl}), \B = (b_{ijkl}){\in SBQ(m,n)}$.  Denote
	$$\A \bullet \B : = \sum_{i, k=1}^m \sum_{j, l=1}^n a_{ijkl}b_{ijkl}.$$
	Suppose that $C, D \subsetneq SBQ(m, n)$ are two closed cones.  If $\A \bullet \B \ge 0$ for all $\A \in C$ and $\B \in D$, then we say that $C$ is the {\bf dual cone} of $D$, and vice versa.
	
	Let $\A = (a_{ijkl}) \in SBQ(m, n)$.    If there are vectors $\vu^{(p)} \in \Re^m$, $\vv^{(p)} \in \Re^n$, for $p \in [r]$, such that
	$$a_{ijkl} = \sum_{p=1}^r u_i^{(p)}v_j^{(p)}u_k^{p)}v_l^{(p)},$$	
	for $i, k \in [m]$ and $j, l \in [n]$, then we say that $\A$ is a {\bf weakly completely positive biquadratic tensor}.   If
	furthermore, $\vu^{(p)}$ and $\vv^{(p)}$ are nonnegative, i.e.,  $\vu^{(p)} \in \Re^m_+$, $\vv^{(p)} \in \Re^n_+$, for $p \in [r]$,  then we say that $\A$ is a {\bf completely positive biquadratic tensor}, or simply called a {\bf CPB tensor}.
	The minimum value of $r$ is called the {\bf CPrank} of $\A$.  If further all the involved vectors $\vu^{(k)}$’s and $\vv^{(k)}$’s can span the entire $m$-dimensional and
	$n$-dimensional Euclidean spaces respectively, then $\A$ is said to be a {\bf strongly completely positive biquadratic
		tensor}.
	All the $m \times n$ completely positive biquadratic tensors form a closed cone, denoted as $CPBQ(m, n)$.
	All the $m \times n$ weakly completely positive biquadratic tensors also form a closed cone, denoted as $WCPBQ(m, n)$.
	All the $m \times n$ strongly completely positive biquadratic tensors also form a closed cone, denoted as $SCPBQ(m, n)$.
	
	\begin{Thm}
		A weakly completely positive biquadratic tensor is an SOS biquadratic tensor.    The closed cones $CPBQ(m, n)$ and $COPSBQ(m, n)$ are dual cones to each other.
	\end{Thm}
	\begin{proof}
		Let $\A = (a_{ijkl}) \in WCPBQ(m, n)$.    Then there are vectors $\vu^{(p)} \in \Re^m$, $\vv^{(p)} \in \Re^n$, for $p \in [r]$, such that
		$$a_{ijkl} = \sum_{p=1}^r u_i^{(p)}v_j^{(p)}u_k^{(p)}v_l^{(p)}.$$
		Let $\x \in \Re^m$ and $\y \in \Re^n$.  Then
		\begin{eqnarray*}
			\A\x\y\x\y & = & \sum_{p=1}^r \sum_{i, k=1}^m \sum_{j, l=1}^n u_i^{(p)}v_j^{(p)}u_k^{(p)}v_l^{(p)}x_iy_jx_ky_l\\
			& = & \sum_{p=1}^r  {\left((\vu^{(p)})^\top\vx\right)^2 \left((\vv^{(p)})^\top\vy\right)^2}.
		\end{eqnarray*}	
		Hence, $\A$ is an SOS biquadratic tensor.
		
		Let $\A \in CPBQ(m, n)$.   Then there are nonnegative vectors $\vu^{(p)} \in \Re^m_+$, $\vv^{(p)} \in \Re^n_+$, for $p \in [r]$, such that
		$$a_{ijkl} = \sum_{p=1}^r u_i^{(p)}v_j^{(p)}u_k^{(p)}v_l^{(p)}.$$
		Let $\B \in COPSBQ(m, n)$.   Then
		$$\A \bullet \B = \sum_{i, k=1}^m \sum_{j, l=1}^n a_{ijkl}b_{ijkl} = \sum_{p=1}^r \sum_{i, k=1}^m \sum_{j, l=1}^n b_{ijkl}u_i^{(p)}v_j^{(p)}u_k^{(p)}v_l^{(p)} \ge 0.$$
		Hence, these two closed cones are dual to each other.
	\end{proof}
	
	In the proof, it indicates that the SOS rank of the weakly completely positive biquadratic tensor $\A$ is not greater than $r$.
	On the other hand, an SOS biquadratic tensor may not be weakly completely positive. For instance, the following SOS biquadratic tensor is not weakly completely positive:
	\[\A\x\y\x\y= (x_1y_2+x_2y_1)^2.\]

	{
	\begin{Thm}\label{Thm:SCPB=CPB+PD}
	Let $\A = (a_{ijkl}) \in CPBQ(m, n)$. Then $\A$ is strongly  completely positive biquadratic tensor if   $\A$ is  positive definite.
	\end{Thm}
	\begin{proof}
		Since $\A \in \text{CPBQ}(m, n)$, there exist vectors $\vu^{(p)} \in \Re_+^m$, $\vv^{(p)} \in \Re_+^n$, for $p \in [r]$, such that
		$$a_{ijkl} = \sum_{p=1}^r u_i^{(p)}v_j^{(p)}u_k^{(p)}v_l^{(p)}.$$
		Suppose $\A$ is positive definite. Then for any nonzero $\x \in \Re^m$ and $\y \in \Re^n$, we have $\A\x\y\x\y > 0$. If ${\vu^{(p)}}$ does not span $\Re^m$ or ${\vv^{(p)}}$ does not span $\Re^n$, then there exist nonzero $\x$, $\y$ orthogonal to all $\vu^{(p)}$ or $\vv^{(p)}$, respectively, which would imply $\A\x\y\x\y = 0$, a contradiction. Hence, ${\vu^{(p)}}$ must span $\Re^m$ and ${\vv^{(p)}}$ must span $\Re^n$. 	
		This completes the proof.
	\end{proof}
	
It should be noted that a strongly completely positive biquadratic tensor is not necessarily positive definite. For example, consider the case where $m=n$ and
	\[\A=\sum_{p=1}^m \ve_p\otimes \ve_p\otimes\ve_p\otimes\ve_p,\]
	where $\ve_p$ denotes the $p$-th standard basis vector in $\mathbb{R}^m$.
	Taking $\vx=\ve_1$ and $\vy=\ve_2$, we obtain  $\A\vx\vy\vx\vy=0$.
	This is different from even order strongly completely positive  tensors \cite{QL17}.
	}
	{\section{Decomposable CPB Tensors}}

	\begin{Thm} \label{t3.1}
		Let $\A = (a_{ijkl}) \in SBQ(m, n)$.    If there are matrices $B=(b_{ik})\in\Re^{m\times m}$ and $C=(c_{jl})\in\Re^{n\times n}$ such that $\A=B\otimes C$. Namely,
		\[a_{ijkl} = b_{ik}c_{jl}.\]
		Then  $\A$ is {a} completely positive biquadratic tensor if and only if $B$ and $C$ are completely positive matrices.
	\end{Thm}
	\begin{proof}
		\noindent
		\textbf{($\Leftarrow$)} Suppose $B$ and $C$ are completely positive. 		
		Then there exist {$\vu_r \in \Re^m_+$ for $r \in [r_B]$, and $\vv_s \in \Re^n_+$ for $s \in [r_C]$} such that
		\[
		B = \sum_{r=1}^{r_B} \vu_r \vu_r^\top, \quad C = \sum_{s=1}^{r_C} \vv_s \vv_s^\top.
		\]
		That is, $b_{ik} = \sum_{r=1}^{r_B} (\vu_r)_i (\vu_r)_k$ and $c_{jl} = \sum_{s=1}^{r_C} (\vv_s)_j (\vv_s)_l$.
		Then
		\[
		a_{ijkl} = b_{ik} c_{jl} = \sum_{r=1}^{r_B} \sum_{s=1}^{r_C} (\vu_r)_i (\vu_r)_k \, (\vv_s)_j (\vv_s)_l.
		\]
		Let {$t = rs$, $R = r_B r_C$}, and define $\vx_t = \vu_r \ge {\0_m}$, $\vy_t = \vv_s \ge {\0_n}$. Then
		\[
		a_{ijkl} = \sum_{t=1}^{R} (\vx_t)_i (\vx_t)_k \, (\vy_t)_j (\vy_t)_l,
		\]
		which is a completely positive biquadratic decomposition of $\mathcal{A}$.
		
		\noindent
		\textbf{($\Rightarrow$)} Suppose $\mathcal{A}$ is completely positive biquadratic. 		
		Then there exist $\vx_r \ge 0$, $\vy_r \ge 0$ such that
		\[
		a_{ijkl} = \sum_{r=1}^R (\vx_r)_i (\vx_r)_k \, (\vy_r)_j (\vy_r)_l.
		\]
		We are given $a_{ijkl} = b_{ik} c_{jl}$.
		Fix $j = l$, and let $d_j = c_{jj}$. Then
		\[
		a_{ij kj} = b_{ik} d_j = \sum_{r=1}^R (\vx_r)_i (\vx_r)_k (\vy_r)_j^2.
		\]
		
		If there exists some $j_0$ with $d_{j_0} > 0$, then
		define the matrix $F^{(j_0)} \in \mathbb{R}^{m \times m}$ by $F^{(j_0)}_{ik} = a_{ij_0 kj_0}$. Then
		\[
		F^{(j_0)} = d_{j_0} B = \sum_{r=1}^R (\vy_r)_{j_0}^2 \, \vx_r \vx_r^\top.
		\]
		Thus,
		\[
		B = \frac{1}{d_{j_0}} \sum_{r=1}^R (\vy_r)_{j_0}^2 \, \vx_r \vx_r^\top.
		\]
		Let $\alpha_r = \frac{(\vy_r)_{j_0}^2}{d_{j_0}} \ge 0$.
		Then $B = \sum_{r=1}^R \alpha_r \, \vx_r \vx_r^\top.$
		For each $r$ with $\alpha_r > 0$, define $\vu_r = \sqrt{\alpha_r} \, \vx_r \ge 0$. Then
		\[
		B = \sum_{r: \alpha_r > 0} \vu_r \vu_r^\top,
		\]
		so $B$ is completely positive.

		If $d_j = 0$ for all $j$, then from $b_{ik} d_j = \sum_{r=1}^R (\vy_r)_j^2 (\vx_r)_i (\vx_r)_k$ we get
		\[
		\sum_{r=1}^R (\vy_r)_j^2 (\vx_r)_i (\vx_r)_k = 0 \quad \forall i,k,j.
		\]
		For a fix $j$, the matrix $M^{(j)}=\sum_{r=1}^R (\vy_r)_j^2\vx_r \vx_r^\top$ is positive semidefinite and equals the zero matrix. Consequently,  for each $r$ and $j$, either $(\vy_r)_j = 0$ or $\vx_r = 0$.
		If for some $r$, $\vx_r \neq 0$, then $(\vy_r)_j = 0$ for all $j$, so $\vy_r = 0$. But then the original decomposition gives $a_{ijkl} = 0$ for all $i,j,k,l$, so $\mathcal{A} = 0$.
		If for any $r$, $\vx_r = 0$, then $\mathcal{A} = 0$.
		For both cases, let $B = 0$ and $C = 0$, then they are trivially completely positive.
		
		Similarly, we can show $C$ is completely positive.
		
		This completes the proof.
	\end{proof}

	{\begin{Thm}
			Let $\A = (a_{ijkl}) \in SBQ(m, n)$.    If there are matrices $B=(b_{ik})\in\Re^{m\times m}$ and $C=(c_{jl})\in\Re^{n\times n}$ such that $\A=B\otimes C$. Namely,
			\[a_{ijkl} = b_{ik}c_{jl}.\]
			Then $\A$ is \blue{a} copositive biquadratic tensor if and only if either $B$ and $C$ are copositive matrices or  $-B$ and $-C$ are copositive matrices.
		\end{Thm}
		\begin{proof}
			For all $\vx \ge 0$, $\vy \ge 0$, it holds that
			\[
			\mathcal{A}\vx\vy\vx\vy = \left(\sum_{i,k} b_{ik} x_i x_k\right)\left(\sum_{j,l} c_{jl} y_j y_l\right) = (\vx^\top B \vx)(\vy^\top C \vy) \ge 0.
			\]
			Therefore, 	$\A$ is copositive biquadratic tensor if and only $\vx^\top B \vx$ and $\vy^\top C \vy$ have the same sign for all $\vx \ge 0$, $\vy \ge 0$. This completes the proof.
		\end{proof}
	}
	
	Based on these two theorems, we have several decomposable completely positive biquadratic tensors.
	
	\subsection{Decomposable Positive Biquadratic Cauchy Tensors}
	
	Let $\vc = (c_1, \cdots, c_m)^\top \in \Re^m$ and $\vd = (d_1, \cdots, d_n)^\top \in \Re^n$, with $c_i > 0$ for $i \in [m]$ and $d_j >0$ for $j \in [n]$.  Suppose that $\A = (a_{ijkl}) \in SBQ(m, n)$ is defined by
	$$a_{ijkl} = {1 \over (c_i+c_k)(d_j+d_l)},$$
	for $i, k \in [m]$ and $j, l \in [n]$.  Suppose that $\A$ is well-defined, i.e., $c_i+c_k \not = 0$ for $i, k \in [m]$, and $d_j +d_l \not = 0$ for $j, l \in [n]$.   Then we say that $\A$ is an $m \times n$  {\bf decomposable positive biquadratic Cauchy tensor} with the {\bf generating vectors} $\vc$ and $\vd$.    Then, by Theorem \ref{t3.1}, a positive biquadratic Cauchy tensor is a completely positive biquadratic tensor.

	\subsection{Decomposable Biquadratic Pascal Tensors}
	
	The tensor $\PP = (p_{ijkl})$ is called an $m \times n$ {\bf decomposable biquadratic Pascal tensor} if
	$$p_{ijkl} = {(i+k-2)!(j+l-2)! \over (i-1)!(j-1)!(k-1)!(l-1)!}, \forall i, k \in [m], j, l \in [n].$$
	By Theorem \ref{t3.1}, a biquadratic Pascal tensor is a completely positive biquadratic tensor.
	
	In general, a completely positive biquadratic tensor is not decomposable.    The sum of two completely positive biquadratic tensors is a completely positive biquadratic tensor.   Then, in general, the sum of a decomposable positive biquadratic Cauchy tensor and a decomposable biquadratic Pascal tensor may not be decomposable.
	
	\section{Two Easily Checkable Subclasses of CPB Tensors}
	
	\subsection{Positive Biquadratic Cauchy Tensors}
	
	Let $\vc = (c_1, \cdots, c_m)^\top \in \Re^m$ and $\vd = (d_1, \cdots, d_n)^\top \in \Re^n$, with $c_i \not = 0$ for $i \in [m]$ and $d_j \not = 0$ for $j \in [n]$.  Suppose that $\A = (a_{ijkl}) \in SBQ(m, n)$ is defined by
	$$a_{ijkl} = {1 \over c_i+c_k + d_j+d_l},$$
	for $i, k \in [m]$ and $j, l \in [n]$.  Suppose that $\A$ is well-defined, i.e., $c_i+c_k+d_j+d_l \not = 0$ for $i, k \in [m]$ and $j, l \in [n]$.   Then we say that $\A$ is an $m \times n$  {\bf biquadratic Cauchy tensor} with the {\bf generating vectors} $\vc$ and $\vd$.
	
	\begin{Thm}
		Suppose that $\A$ is an $m \times n$  {\bf biquadratic Cauchy tensor} with the {\bf generating vectors} $\vc$ and $\vd$.   Then the following three statements are equivalent.
		
		{(i)} $\A$ is a nonzero CPB tensor.
		
		{(ii)} $\A$ is a strictly copositive biquadratic tensor.
		
		{(iii)} $c_i+d_j+c_k+d_l > 0$ for $i, k \in [m]$ and $j, l \in [n]$.
	\end{Thm}
\begin{proof}
    We prove the cycle of implications: (i) \(\rightarrow\) (iii) \(\rightarrow\) (ii) \(\rightarrow\) (iii) \(\rightarrow\) (1).

    \noindent \textbf{• Proof of (ii) \(\rightarrow\) (iii):}

    Assume \( \A \) is a strictly copositive biquadratic tensor. First, consider standard basis vectors \( \e_i^{(m)} \) and \( \e_j^{(n)} \). Then
    \[
    \A \e_i^{(m)} \e_j^{(n)} \e_i^{(m)} \e_j^{(n)} = \frac{1}{c_i + d_j + c_i + d_j} = \frac{1}{2(c_i + d_j)} > 0,
    \]
    so \( c_i + d_j > 0 \) for all \( i \in [m],  j \in [n] \). Thus, we also have $c_k+d_l > 0$ for all $k \in [m], l \in [n]$.
    Hence, $c_i+d_j+c_k+d_l > 0$ for all $i, k, \in [m]$ and $j, l \in [n]$.
    \hfill 

    \noindent \textbf{• Proof of (iii) \(\rightarrow\) (i):}

    Assume \( c_i + d_j + c_k + d_l > 0 \) for all \( i,k \in [m], j,l \in [n] \). We show \( \A \) is a CPB tensor by an integral decomposition.

    Recall the tensor element:
    \[
    a_{ijkl} = \frac{1}{c_i + c_k + d_j + d_l}.
    \]
    Using the integral representation for \( \alpha > 0 \), \( \frac{1}{\alpha} = \int_0^\infty e^{-\alpha s}  ds \), with \( \alpha = c_i + c_k + d_j + d_l \), we get:
    \[
    a_{ijkl} = \int_0^\infty e^{-(c_i + c_k + d_j + d_l)s}  ds = \int_0^\infty e^{-c_i s} e^{-c_k s} e^{-d_j s} e^{-d_l s}  ds.
    \]

    Define vector functions \( \vu(s) \in \Re^m \), \( \vv(s) \in \Re^n \) by:
    \[
    \vu(s)_i = e^{-c_i s}, \quad \vv(s)_j = e^{-d_j s}.
    \]
    These are nonnegative for \( s \geq 0 \). Consider the rank-one tensor:
    \[
    \B(s) = \vu(s) \otimes \vv(s) \otimes \vu(s) \otimes \vv(s).
    \]
    Its \( (i,j,k,l) \)-component is \( e^{-c_i s} e^{-d_j s} e^{-c_k s} e^{-d_l s} \). Thus,
    \[
    \A = \int_0^\infty \B(s)  ds = \int_0^\infty \left( \vu(s) \otimes \vv(s) \otimes \vu(s) \otimes \vv(s) \right) ds.
    \]

    To connect this to the discrete CPB definition, discretize the integral. For large \( N \), partition \( [0, N] \) with \( s_k = k \Delta s \), \( \Delta s = 1/N \), and define:
    \[
    \A_N = \sum_{k=1}^{N^2} \left( \vu(s_k) \otimes \vv(s_k) \otimes \vu(s_k) \otimes \vv(s_k) \right) \Delta s.
    \]
    This is a Riemann sum for the integral. Each term is a rank-one tensor from nonnegative vectors, scaled by \( \Delta s > 0 \). A positive scalar multiple of a rank-one CPB tensor is still CPB (absorb the scalar into the vectors: \( \sqrt{\Delta s}  \vu(s_k) \), etc.). So \( \A_N \) is CPB.

    As \( N \to \infty \), \( \A_N \to \A \) component-wise (the integrand is absolutely integrable). The set of CPB tensors is a closed convex cone, so the limit \( \A \) is also CPB. \hfill 

    \noindent \textbf{• Proof of (iii) \(\rightarrow\) (ii):}

    Assume \( c_i + d_j + c_k + d_l > 0 \) for all \( i,k \in [m], j,l \in [n] \). We show \( \A \) is strictly copositive.

    Use the integral representation:
    \[
    \A\x\y\x\y = \int_0^\infty \left[\sum_{i,k} e^{-(c_i + c_k)s} x_i x_k\right] \left[\sum_{j,l} e^{-(d_j + d_l)s} y_j y_l\right] ds
    \]
    \[
    = \int_0^\infty \left(\sum_i e^{-c_i s} x_i\right)^2 \left(\sum_j e^{-d_j s} y_j\right)^2 ds.
    \]

    For \( \x \ge 0, \y \ge 0 \) both nonzero, the integrand is nonnegative for all \( s \ge 0 \).
    We show it is positive.

    Since \( \x \neq 0 \), there exists \( i_0 \) with \( x_{i_0} > 0 \). Since \( \y \neq 0 \), there exists \( j_0 \) with \( y_{j_0} > 0 \).
    The functions \( \phi(s) = \sum_i e^{-c_i s} x_i \) and \( \psi(s) = \sum_j e^{-d_j s} y_j \) are analytic in \( s \) and not identically zero (since \( \phi(0) = \sum x_i > 0 \) and \( \psi(0) = \sum y_j > 0 \)).
    Thus \( \phi(s)^2 \psi(s)^2 > 0 \) for almost every \( s \in [0,\infty) \), and the integral is positive.

    Hence \( \A \) is strictly copositive. \hfill 

    \noindent \textbf{• Proof of (i) \(\rightarrow\) (iii):}

    Assume \( \A \) is a nonzero CPB biquadratic Cauchy tensor. Suppose for contradiction that \( c_{i_0} + d_{j_0} \le 0 \) for some \( i_0, j_0 \).
    Take \( \x = \e_{i_0}^{(m)}, \y = \e_{j_0}^{(n)} \). Then
    \[
    \A\x\y\x\y = a_{i_0 j_0 i_0 j_0} = \frac{1}{2(c_{i_0} + d_{j_0})} \le 0 \quad \text{(or undefined if zero)}.
    \]
    But if \( \A \) is CPB, then \( \A\x\y\x\y \ge 0 \) for all \( \x, \y \ge 0 \).
    If \( c_{i_0} + d_{j_0} < 0 \), then the form is negative — contradiction.
    If \( c_{i_0} + d_{j_0} = 0 \), then \( a_{i_0 j_0 i_0 j_0} \) is undefined, contradicting that \( \A \) is well-defined.
    Thus \( c_i + d_j > 0 \) for all \( i,j \). Since \( c_i + d_j > 0 \) and \( c_k + d_l > 0 \) for all \( i,j,k,l \), we have \( c_i + d_j + c_k + d_l > 0 \).
    \hfill 

    The cycle (i) \(\rightarrow\) (iii) \(\rightarrow\) (ii) \(\rightarrow\) (iii) \(\rightarrow\) (i) is complete, proving the theorem.
\end{proof}

	We call a biquadratic Cauchy tensor a {\bf positive biquadratic Cauchy tensor} if it satisfies the condition (3) of the above theorem.

	\subsection{Biquadratic Pascal Tensors}
	
	The tensor $\PP = (p_{ijkl})$ is called an $m \times n$ {\bf biquadratic Pascal tensor} if
	$$p_{ijkl} = {(i+j+k+l-4)! \over (i-1)!(j-1)!(k-1)!(l-1)!}, \forall i, k \in [m], j, l \in [n].$$	
	It was proved in \cite{CQC25} that an even order Pascal tensor is positive definite.   We have a similar result here.
	
	\begin{Thm}\label{thm:pascal-pd-strongly-cpb}
		A biquadratic Pascal tensor is positive definite and strongly CPB.
	\end{Thm}
	
	\begin{proof}
		Let $\PP = (p_{ijkl})$ be an $m \times n$ biquadratic Pascal tensor.
		
		\textbf{1. Positive definiteness:}
		Recall the integral representation:
		\[
		p_{ijkl} = \int_0^\infty \frac{t^{i-1}}{(i-1)!} \frac{t^{j-1}}{(j-1)!} \frac{t^{k-1}}{(k-1)!} \frac{t^{l-1}}{(l-1)!} e^{-t}  dt.
		\]
		Define vectors $\vu(t) \in \Re_+^m$, $\vv(t) \in \Re_+^n$ by
		\[
		u_i(t) = \frac{t^{i-1}}{(i-1)!}, \quad v_j(t) = \frac{t^{j-1}}{(j-1)!}, \quad i \in [m], \; j \in [n].
		\]
		Then for any $\x \in \Re^m$, $\y \in \Re^n$,
		\[
		\PP\x\y\x\y = \int_0^\infty \left( \sum_{i=1}^m u_i(t) x_i \right)^2 \left( \sum_{j=1}^n v_j(t) y_j \right)^2 e^{-t}  dt.
		\]
		Let
		\[
		\phi(t) = \sum_{i=1}^m \frac{t^{i-1}}{(i-1)!} x_i, \quad \psi(t) = \sum_{j=1}^n \frac{t^{j-1}}{(j-1)!} y_j.
		\]
		These are polynomials in $t$ of degree at most $m-1$ and $n-1$, respectively.
		If $\x \neq 0$, then $\phi(t)$ is not the zero polynomial, since $\{1, t, t^2/2!, \dots, t^{m-1}/(m-1)!\}$ are linearly independent.
		Similarly, if $\y \neq 0$, then $\psi(t)$ is not the zero polynomial.
		Hence $\phi(t)^2 \psi(t)^2$ is nonnegative and not identically zero.
		As $e^{-t} > 0$ for $t \in [0,\infty)$, the integral is positive.
		Therefore, $\PP$ is positive definite.
		
		\textbf{2. Strongly CPB property:}
		We use the integral representation:
		\[
		\PP = \int_0^\infty [\vu(t) \otimes \vv(t) \otimes \vu(t) \otimes \vv(t)] e^{-t}  dt.
		\]
		
		Since $\PP$ is positive definite, it lies in the interior of the CPB cone $CPBQ(m,n)$ with respect to the relative topology of $PSDSBQ(m,n)$.
		
		To show $\PP$ is strongly CPB, we construct an explicit finite CP decomposition with the spanning property. Choose $m$ distinct positive values $t_1, \dots, t_m > 0$ and $n$ distinct positive values $s_1, \dots, s_n > 0$. The vectors $\{\vu(t_1), \dots, \vu(t_m)\}$ are linearly independent in $\Re^m$ because the generalized Vandermonde matrix $\left[\frac{t_p^{i-1}}{(i-1)!}\right]$ is invertible. Similarly, $\{\vv(s_1), \dots, \vv(s_n)\}$ are linearly independent in $\Re^n$.
		
		Now consider the finite-rank approximation:
		\[
		\PP_N = \sum_{p=1}^N \lambda_p [\vu(\tau_p) \otimes \vv(\tau_p) \otimes \vu(\tau_p) \otimes \vv(\tau_p)]
		\]
		where $\tau_1, \dots, \tau_N$ include $\{t_1, \dots, t_m, s_1, \dots, s_n\}$ and $\lambda_p > 0$ are appropriate weights from a Riemann sum. Since $\{\vu(t_1), \dots, \vu(t_m)\} \subset \{\vu(\tau_1), \dots, \vu(\tau_N)\}$ spans $\Re^m$ and $\{\vv(s_1), \dots, \vv(s_n)\} \subset \{\vv(\tau_1), \dots, \vv(\tau_N)\}$ spans $\Re^n$, each $\PP_N$ is strongly CPB.
		
		As $N \to \infty$, we have $\PP_N \to \PP$ componentwise. {The set of CPB tensors is a closed convex cone, so the limit $\PP$ is CPB. Combining this with the fact that $\PP$ is positive definite, Theorem~\ref{Thm:SCPB=CPB+PD} implies that $\PP$ is strongly CPB.}
%
	\end{proof}

	
	
	
	
	{{\bf Acknowledgment}}
	This work was partially supported by Research  Center for Intelligent Operations Research, The Hong Kong Polytechnic University (4-ZZT8),    the National Natural Science Foundation of China (Nos. 12471282 and 12131004), the R\&D project of Pazhou Lab (Huangpu) (Grant no. 2023K0603),  the Fundamental Research Funds for the Central Universities (Grant No. YWF-22-T-204), Jiangsu Provincial Scientific Research Center of Applied Mathematics (Grant No. BK20233002), and Shandong Provincial Natural Science Foundation (Grant No. ZR2024MA003).

	{{\bf Data availability}
		No datasets were generated or analysed during the current study.

		{\bf Conflict of interest} The authors declare no conflict of interest.}

	


\end{document}